\documentclass[10pt]{article}
\usepackage[T1]{fontenc}
\usepackage[english]{babel}
\usepackage{lmodern}
\usepackage{mathtools}
\usepackage{calc}
\usepackage{amsthm}
\usepackage{amssymb}
\usepackage{amsfonts}
\usepackage{enumitem}
\usepackage[ocgcolorlinks,unicode]{hyperref}
\usepackage{caption}
\usepackage{polynom}
\allowdisplaybreaks

\newcommand{\E}{\ensuremath{\mathcal{E}}}
\newcommand{\F}{\ensuremath{\mathcal{F}}}
\newcommand{\GG}{\ensuremath{\mathcal{G}}}
\newcommand{\V}{\ensuremath{\mathcal{V}}}
\newcommand{\WWW}{\ensuremath{\mathfrak{W}}}
\newcommand{\W}{\ensuremath{\mathcal{W}}}
\newcommand{\R}{\ensuremath{\mathcal{R}}}
\newcommand{\NN}{\ensuremath{\mathbb{N}}}

\renewcommand{\P}{\ensuremath{\mathcal{P}}}

\newcommand{\Gauss}[2]{\ensuremath{\genfrac{[}{]}{0pt}{0}{#1}{#2}}}
\newcommand{\gauss}[2]{\ensuremath{\genfrac{[}{]}{0pt}{1}{#1}{#2}}}
\newcommand{\PG}{\text{PG}}
\newcommand{\rank}{\text{rank}}

\theoremstyle{plain}
\newtheorem{Theorem}{Theorem}[section]
\newtheorem{Lemma}[Theorem]{Lemma}
\newtheorem{Conjecture}[Theorem]{Conjecture}
\newtheorem{Proposition}[Theorem]{Proposition}
\newtheorem{Remark}[Theorem]{Remark}

\theoremstyle{definition}

\newtheorem{Notation}[Theorem]{Notation}
\newtheorem{Example}[Theorem]{Example}

\title{On the Chromatic Number of some generalized Kneser Graphs}

\author{Jozefien D'haeseleer\thanks{Department of Mathematics: Analysis, Logic and Discrete Mathematics, Ghent University, Krijgslaan 281, Building S8, 9000 Gent, Flanders, Belgium}\and Klaus Metsch\thanks{Justus-Liebig-Universit\"at, Mathematisches Institut, Arndtstra\ss e 2, D-35392 Gie\ss en}\and Daniel Werner\footnotemark[2]}

\date{\today}

\begin{document}
\maketitle

\begin{abstract}
 We determine the chromatic number of the Kneser graph $q\Gamma_{7,\{3,4\}}$ of flags of vectorial type $\{3,4\}$ of a rank $7$ vector space over the finite field $GF(q)$ for large $q$ and describe the colorings that attain the bound. This result relies heavily, not only on the independence number, but also on the structure of all \emph{large} independent sets. Furthermore, our proof is more general in the following sense: it provides the chromatic number of the Kneser graphs $q\Gamma_{2d+1,\{d,d+1\}}$ of flags of vectorial type $\{d,d+1\}$ of a rank $2d+1$ vector space over $GF(q)$ for large $q$ as long as the \emph{large} independent sets of the graphs are only the ones that are known.
\end{abstract}

\begin{section}{Introduction}
The introduction is split into two parts. In the first part we introduce most of the required notation and in the second part we state our main results and give a layout of the strategy of our proof.

\begin{subsection}{Notation}
	Let $V$ be a vector space of some dimension $n\in\NN$. A \emph{flag} of $V$ is a set $f$ of non-trivial proper subspaces of $V$ such that $U\le W$ or $W\le U$ for all $U,W\in f$. The \emph{vectorial type} of a \emph{flag} $f$ is the set $\{\rank(U)\mid U\in f\}$ and as such a subset of $\{1,2,\dots,n-1\}$. Furthermore, two flags $f_1$ and $f_2$ are said to be in \emph{general position}, if $U_1\cap U_2=\{0\}$ or $U_1+U_2=V$ for all $U_1\in f_1$ and $U_2\in f_2$.
	
	Given $n\ge 3$, a subset $J\subseteq \{1,\dots,n-1\}$ and a finite field $GF(q)$, the \emph{$q$-Kneser graph} $qK_{n,J}$ is the graph whose vertex set is the set of all flags of vectorial type $J$ of the vector space $GF(q)^n$ and in which two vertices are adjacent if they are in general position. Note that we only consider finite graphs in this work.

	Now, let $\Gamma$ be a graph with vertex set $X$. An \emph{independent set} of $\Gamma$ is a set of pairwise non-adjacent vertices of the graph. The \emph{independence number} $\alpha(\Gamma)$ of $\Gamma$ is the cardinality of its largest independent sets. A \emph{coloring} of $\Gamma$ is a map $g:X\to C$ such that $g^{-1}(c)$ is an independent set for all $c\in C$. The smallest cardinality of any set $C$ such that there exists a coloring $g:X\to C$ is called the \emph{chromatic number} of $\Gamma$ and is denoted by $\chi(\Gamma)$. Clearly $\chi(\Gamma)$ is the smallest integer $\chi$ such that $X$ is the union of $\chi$ maximal independent sets.

	Finally, let $V$ be a vector space of rank $2d+1\ge3$ over $GF(q)$. For every rank 1 subspace $P$ we denote by $F(P)$ the set of all flags of $V$ of vectorial type $\{d,d+1\}$ whose rank $d$ subspace contains $P$ and call this set a \emph{point-pencil}. Dually for every rank $2d$ subspace $H$ of $V$ we denote by $F(H)$ the set of all flags of vectorial type $\{d,d+1\}$ whose rank $d+1$ subspace is contained in $H$ and call this set a \emph{dual point-pencil}.
	
	Notice that point-pencils and dual point-pencils are independent sets of cardinality $\approx q^{d^2+d-1}$ but they are not maximal independent sets. In fact, for $d=2$ every maximal independent set containing a point-pencil or a dual point-pencil has cardinality $\alpha(q\Gamma_{2d+1,\{d,d+1\}})$, see \cite{Jozefien&Daniel&Klaus}. However, for $d\ge 3$ this is no longer true. There are different maximal independent sets containing a point-pencil or a dual point-pencil and they do not all have the same size. Nevertheless, the structure of these examples can still be described quite precisely (as we explain in Section \ref{Sec_Examples}).
\end{subsection}

\begin{subsection}{Results and Strategy of Proof}
	We determine the chromatic number of the Kneser graph $qK_{7,\{3,4\}}$ for large values of $q$. In order to give a layout of the strategy of our approach we consider a graph $\Gamma$ with vertex set $X$.

	Clearly, for any coloring $g:X\to C$ the set $C$ satisfies the bound $|C|\ge\frac{|X|}{\alpha(\Gamma)}$. Now, the key tool in our proof is, that there is some value $\alpha'<\alpha(\Gamma)$ such that structural information on any independent set of $\Gamma$ of size larger than $\alpha'$ is known. In that situation, if $g:X\to C$ is a coloring of $\Gamma$ such that $\alpha'\cdot|C|\ll|X|$, then at least $\frac{|X|-\alpha'\cdot|C|}{\alpha(\Gamma)-\alpha'}$ color classes of $g$ have cardinality larger than $\alpha'$ and satisfy the given structural conditions. We use this structural information to provide a lower bound on $|C|$ that coincides with the cardinality of a known coloring of $\Gamma$ and thus in fact determine $\chi(\Gamma)$.

	Note that a similar approach was successfully applied for many Kneser graphs $q\Gamma_{n,J}$ with $|J|=1$ in \cite{blokhuis_neu,blokhuis3} as well as in one special case for $|J|=2$ in \cite{Jozefien&Daniel&Klaus}. However, for many Kneser graphs $q\Gamma_{n,J}$, with $|J|\ge2$ the independence number is not known and only in very few cases both the independence number as well as structural information on independent sets of maximal size is given. One reason for the lack of this structural information is, that in some cases the independence number was determined by algebraic arguments and these do not automatically give the structure of the largest independent sets, see for example the recent work \cite{Sam}.

	Now, for the case covered in this work, that is in the Kneser graph $q\Gamma_{7,\{3,4\}}$, the required structural information as well as the bound $\alpha'$ is provided in \cite{Daniel&Klaus}.
	Furthermore, we remark that the proof provided here does not only cover that case, but also every graph $q\Gamma_{2d+1,\{d,d+1\}}$, for which an analogous result to that given for $d=3$ in \cite{Daniel&Klaus} holds.
	We state this requirement more precisely in the following conjecture.

	\begin{Conjecture}\label{conject}
		For every integer $d\ge 2$ there is an integer $\rho(d)$ such that every maximal independent set of the Kneser graph $q\Gamma_{2d+1,\{d,d+1\}}$ contains a point-pencil, a dual point-pencil, or has at most $\rho(d)\cdot q^{d^2+d-2}$ elements.
	\end{Conjecture}

	Note that it has been shown that this conjecture holds true for $d=2$, which was implicitly proven in \cite{blokhuis1}, as well as for $d=3$, as is shown in \cite{Daniel&Klaus}.

	Using this conjecture we may now state our main result as follows.

	\begin{Theorem}\label{main1}
		If Conjecture \ref{conject} holds true for some integer $d\ge 3$, then
		\begin{align*}
			\chi(q\Gamma_{2d+1,\{d,d+1\}})=\frac{q^{d+2}-1}{q-1}-q
		\end{align*}
		provided $q>3\cdot 112^{2^{d+1}-1}\cdot2^{-d-1}$ and $q\ge\frac{3}{2}\alpha^2+\frac{21}{2}\alpha+17$ where $\alpha=\max\{5,\rho(d)\}$. Moreover, if $\F$ is a family of this many maximal independent sets that cover the vertex set, then --- up to duality --- there exists a rank $d+2$ subspace $U$ of the underlying vector space and an injective map $\mu$ from $\F$ to the set of rank 1 subspaces of $U$ such that $F(\mu(C))\subseteq C$ for all $C\in\F$.
	\end{Theorem}

Since Conjecture \ref{conject} holds true for $d=3$, we find the following corollary.

	\begin{Theorem}\label{main2}
		For $q>3\cdot 7^{15}\cdot 2^{56}$ we have $\chi(q\Gamma_{7,\{3,4\}})=q^4+q^3+q^2+1$.
	\end{Theorem}

	Finally, we note that, regardless of whether or not Conjecture \ref{conject} holds true for some integer $d$, in the next section we will see that $q\Gamma_{2d+1,\{d,d+1\}}$ can be covered with $\frac{q^{d+2}-1}{q-1}-q$ maximal independent sets in different ways. In fact, we thus have $\chi(q\Gamma_{2d+1,\{d,d+1\}})\le\frac{q^{d+2}-1}{q-1}-q$ for all integers $d$ and, if this holds with equality, then there are different optimal colorings of this graph. We provide examples of these colorings as well as some structural information on maximal independent sets in the next section.
\end{subsection}
\end{section}

\begin{section}{Independent sets and colorings}\label{Sec_Examples}
In order to gain geometric intuition we switch to projective language, that is we pass from vector spaces of rank $2d+1$ over the finite field $GF(q)$ to the projective space $\PG(2d,q)$ of (projective) dimension $2d$. Note that throughout this work we use \emph{dimension} whenever we refer to projective dimension and rank whenever we refer to the rank of a vector space.

In this setting the Kneser graph $q\Gamma_{2d+1,\{d,d+1\}}$ is isomorphic to the graph $\Gamma_d(q)$, which we define as follows. The vertices of $\Gamma_d(q)$ are the pairs $(\pi,\tau)$ of subspaces $\pi$ and $\tau$ of $\PG(2d,q)$ of respective dimensions $d-1$ and $d$ with $\pi\subseteq\tau$. Two vertices $(\pi,\tau)$ and $(\pi',\tau')$ of $\Gamma_d(q)$ are adjacent if and only if $\pi\cap\tau'=\pi'\cap\tau=\emptyset$, where $\emptyset$ is the empty subspace of $\PG(2d,q)$. The vertex set of $\Gamma_d(q)$ will be denoted by $X(\Gamma_d(q))$. The vertices $(\pi,\tau)$ of this graph are called flags, too, and their (projective) type is $\{d-1,d\}$. Note that in the following we will always refer to this (projective) type, when we say type.

\begin{Example}[Independent sets of $\Gamma_d(q)$]\label{Examples_IndependentSets}\
	\begin{enumerate}
		\item\label{Example_One}
			For a point $P$ and a set $\mathcal{U}$ of $d$-dimensional subspaces through $P$, such that for all $\tau,\tau'\in\mathcal{U}$ we have $\dim(\tau\cap\tau')\ge1$, we define
			\begin{align*}
				F(P,\mathcal{U}):=\{(\pi,\tau)\in X(\Gamma_d(q))\mid P\in \pi\text{ or }\tau\in\mathcal{U}\}.
			\end{align*}
			We call $\{(\pi,\tau)\in F(P,\mathcal{U})\mid P\in \pi\}$ the \emph{generic part} and $\{(\pi,\tau)\in F(P,\mathcal{U})\mid P\notin \pi\}$ the \emph{special part} of $F(P,\mathcal{U})$. We also say that $F(P,\mathcal{U})$ is \emph{based on the point $P$} and call $P$ the \emph{base point} of $F(P,\mathcal{U})$.

			If there exists a line $\ell$ on $P$ such that $\mathcal{U}$ consists of all $d$-dimensional subspaces $\tau$ with $\ell\subseteq \tau$, then we denote $F(P,\mathcal{U})$ also by $F(P,\ell)$ and say that the special part of this set is \emph{based on} the line $\ell$.

			If there exits a hyperplane $H$ on $P$ such that $\mathcal{U}$ consists of all $d$-dimensional subspaces $\tau$ with $P\in \tau\subseteq H$, then we denote $F(P,\mathcal{U})$ also by $F(P,H)$ and say that the special part of this set is \emph{based on} the hyperplane $H$.
		\item\label{Example_Two}
			Dually, for a hyperplane $H$ and a set $\E$ of subspaces of dimension $d-1$ of $H$ with pairwise non-empty intersection, we define
			\begin{align*}
				F(H,\E):=\{(\pi,\tau)\in X(\Gamma_d(q))\mid \tau\subseteq H\text{ or }\pi\in\E\}.
			\end{align*}
			We call $\{(\pi,\tau)\in F(H,\E)\mid \tau\subseteq H\}$ the \emph{generic part} and $\{(\pi,\tau)\in F(H,\E)\mid \tau\not\subseteq H\}$ the \emph{special part} of $F(H,\E)$. We also say that $F(H,\E)$ is \emph{based on} the hyperplane $H$.
	\end{enumerate}
\end{Example}

In order to state the cardinality of these sets we use the \emph{Gaussian binomial coefficient}
\begin{align*}
	\Gauss{a}{b}_q:=\prod_{i=1}^b\frac{q^{a-b+i}-1}{q^i-1}
\end{align*}
for integers $q\ge 2$ and $a\ge b\ge 0$. Whenever $q$ is clear from the context we omit the index $q$ and set $\theta_d:=\gauss{d+1}{1}$.

\begin{Lemma} \label{Sizeofindependentsets}
	With the notation of Example \ref{Examples_IndependentSets} we have:
	\begin{enumerate}[label=(\alph*)]
		\item
			$F(P,\mathcal{U})$ is an independent set of $\Gamma_d(q)$. Its generic part has cardinality $\gauss{2d}{d+1}\theta_d$ and its special part has cardinality $|\mathcal{U}|q^d$.
		\item
			If the special part of $F(P,\mathcal{U})$ is based on a line or a hyperplane, then $|\mathcal{U}|=\gauss{2d-1}{d-1}$. Furthermore, if $F(P,\mathcal{U})$ is maximal but its special part is based neither on a line or a hyperplane, then
			\begin{align}\label{boundU}
				|\mathcal{U}|<(1+q^{-1})\theta_{d-2}\theta_{d-1}^{d-1}.
			\end{align}
	\end{enumerate}
\end{Lemma}
\begin{proof}
	\begin{enumerate}[label=(\alph*)]
		\item
			It is immediate that $F(P,\mathcal{U})$ is an independent set of $\Gamma_d(q)$. To verify the cardinality of its generic part we notice that the number of $(d-1)$-dimensional subspaces on $P$ is $\gauss{2d}{d+1}$ and that each of these subspaces lies in $\theta_d$ subspaces of dimension $d$. For the cardinality of the special part we notice that each subspace $\tau\in \mathcal{U}$ contains exactly $q^d$ subspaces $\pi$ of dimension $d-1$ with $P\notin\pi$.
		\item
			If the special part is based on a line, then its cardinality is the number of $d$-subspaces on a line, which is $\gauss{2d-1}{d-1}$. Furthermore, if it is based on a hyperplane $H$ on $P$, then its cardinality is the number of $d$-subspaces on $P$ in $H$, which is the same number. Now suppose that $F(P,\mathcal{U})$ is maximal but its special part is based neither on a line nor a hyperplane. Consider a $d$-subspace $\tau$ on $P$ that does not lie in $\mathcal{U}$ and let $\pi$ be a $(d-1)$-subspace of $\tau$ with $P\notin\pi$. Then $(\pi,\tau)\notin F(P,\mathcal{U})$. The maximality of $F(P,\mathcal{U})$ implies that $(\pi,\tau)$ is in general position to some $(\pi',\tau')\in F(P,\mathcal{U})$. Then $\pi'\cap \tau=\emptyset$ and hence $P\notin\pi'$. Thus $(\pi',\tau')$ lies in the special part of $F(P,\mathcal{U})$ and thus $\tau'\in\mathcal{U}$. As $(\pi,\tau)$ and $(\pi',\tau')$ are in general position, this implies that $\tau\cap \tau'=P$. Hence $\mathcal{U}$ is a maximal set of $d$-subspaces on $P$ such that any two subspaces of $\mathcal{U}$ intersect in at least a line. Finally, a result of Blokhuis, Brouwer and Sz\H{o}nyi \cite[Section 3]{blokhuis3} applied to $\mathcal{U}$ in the quotient space $PG(2d,q)/P$ implies (\ref{boundU}).\qedhere
	\end{enumerate}
\end{proof}

By definition, the chromatic number of a graph is the smallest cardinality of a partition of its vertex set into independent sets, but of course it is also the smallest cardinality of a cover of its vertex set by independent sets. We now provide examples of such coverings.

\begin{Example}[Coverings of $X(\Gamma_d(q))$ by independent sets]\label{ExampleCoverings}
	Let $U\le\PG(2d,q)$ be a subspace of dimension $d+1$, consider a set $W$ of $q$ points of $U$ and let $L$ be the set of lines of $U$ that meet $W$. Furthermore, suppose that there exists an injective map $\nu:L\to U\setminus W$ such that $\nu(l)\in l$ for all $l\in L$. Then
	\begin{align*}
		\{F(\nu(l),l)\mid l\in L\}\cup\{F(P,\emptyset)\mid P\in U\setminus (\nu(L)\cup W)\}
	\end{align*}
	is a set of independent sets of $\Gamma_d(q)$ whose union contains all vertices of $\Gamma_d(q)$.
\end{Example}

\begin{Remark}\label{labelnewbyklaus24}
	\begin{enumerate}[label=(\alph*)]
		\item
			There are different possibilities for $(W,\nu)$ satisfying the required condition in Example \ref{ExampleCoverings} and we provide an example.
			
			Let $P_0,\dots,P_q$ be the points of a line $\ell\le U$ and set $W=\{P_1,\dots,P_q\}$. In each plane $\pi$ of $U$ on $\ell$ there are $q$ lines through $P_0$ distinct from $\ell$ and thus there is a bijective map $h_\pi$ from $W$ to the set of these lines. Now we may define $\nu$ by $\nu(\ell)=P_0$ and $\nu(l):=l\cap h_{\langle \ell,l\rangle}(l\cap\ell)$ for $l\in L\setminus\{\ell\}$. The pair $(W,\nu)$ has the properties required in Example \ref{ExampleCoverings} and moreover satisfies $U=\nu(L)\cup W$.
			
			It is also possible to construct pairs $(W,\nu)$ satisfying $U\not=\nu(L)\cup W$, for example for odd $q\ge 5$ when $W$ consists of $q$ points of a conic in a plane of $U$, but we omit the details.
		\item
			We can find different coverings with independent sets by dualizing the coverings described in Example \ref{ExampleCoverings} and part (a) of this remark.
		\item
			Since there are $\Gauss{d+2}{1}_q-q$ independent sets in the given coverings, we find
			\begin{align*}
				\chi(\Gamma_d(q))\le \Gauss{d+2}{1}_q-q.
			\end{align*}
    This was already noticed for $d=2$ in \cite{Jozefien&Daniel&Klaus}.
	\end{enumerate}
\end{Remark}
\end{section}

\begin{section}{Preliminaries}
This section contains a result on point sets in Proposition \ref{4-space_1} as well as some technical bounds in Lemmas \ref{LemmaWithBounds} and \ref{newlemma}, that we will be needed later on.

\begin{Lemma}
 Consider a projective space $\P$ of order $q$, a set $M$ of points of $\P$ and points $P_1,\dots,P_{s+1}$ of $\P$, $s\ge 0$, such that $\langle P_1,\dots,P_{s+1}\rangle$ is a subspace of dimension $s$ with no point in $M$. Let $\mu$ be an upper bound on the number of lines on $P_{s+1}$ that meet $M$. Let $0<c\in\mathbb{R}$ and let $\V$ be a set of $s$-dimensional subspaces such that $P_1,\dots,P_s\in V$ and $|V\cap M|\ge cq^s$ for all $V\in\V$.

 Then for every real number $\gamma$ with $0<\gamma<1$, there exist at least $\frac{1-\gamma}q|\V|$ subspaces $W$ of dimension $s+1$ satisfying $P_1,\dots,P_{s+1}\in W$ and $|W\cap M|>\frac{\gamma}{\mu}c^2q^{2s}|\V|$.
\end{Lemma}
\begin{proof}
	We may assume that $\V\not=\emptyset$. For $V\in \V$ we have $V\cap M\not=\emptyset$ and hence $P_{s+1}\notin V$. Put $x:=\gamma\frac{|\V|cq^s}{\mu}$,
	\begin{align*}
		\WWW
		&:=\{\langle{V,P_{s+1}\rangle}:V\in\V\},\\
		\W
		&:=\left\{W\in\WWW:|\{V\in\V:V\subseteq W\}|>x\right\},
	\end{align*}
	and $\overline{\W}:=\WWW\setminus\W$. Note that the elements of $\WWW$ are subspaces of dimension $s+1$. Now, for $W\in\WWW$ we have $W=\langle V,P_{s+1}\rangle$ for some $V\in\V$ and $P_{s+1}$ lies on at least $|V\cap M|\ge cq^s$ lines of $W$ which meet $M$. Furthermore, if $W$ and $W'$ are distinct elements of $\WWW$ and $l$ is a line on $P_{s+1}$ with $l\subseteq W,W'$, then $l\subseteq W\cap W'=\langle P_1,\dots,P_{s+1}\rangle$ and thus $l\cap M=\emptyset$. Since $\mu$ is an upper bound on the number of lines on $P_{s+1}$ which meet $M$, this proves that $|\WWW|\le\frac{\mu}{c}q^{-s}$. Since every element of $\overline{\W}$ contains at most $x$ elements of $\V$, it follows that
	\begin{align*}
		|\{V\in\V:\langle V,P_{s+1}\rangle\in\overline{\W}\}|\le\frac{\mu}{c}q^{-s}\cdot x=\gamma |\V|
	\end{align*}
	and hence $\langle V,P_{s+1}\rangle\in \W$ for least $(1-\gamma)|\V|$ elements of $\V$. Since every subspace $W\in\W$ contains at most $q$ subspaces $V\in\V$, we find $|\W|\ge(1-\gamma)|\V|/q$. Finally, since distinct elements $V$ and $V'$ of $\V$ satisfy $(V\cap V')\cap M=\langle P_1,\hdots, P_s\rangle \cap M=\emptyset$, we see that every $W\in\W$ satisfies
	\begin{align*}
		|W\cap M|> x\cdot cq^s=\frac{\gamma}{\mu}c^2q^{2s}|\V|.\tag*{\qedhere}
	\end{align*}
\end{proof}

\begin{Proposition}\label{4-space_1}
	Suppose that $M$ is a set of points in $\PG(2d,q)$ and there are $d+1$ points $P_1,P_2,\dots P_{d+1}$ that span a $d$-dimensional subspace $\tau$ with $\tau\cap M=\emptyset$. Furthermore, let $n_0$ and $d_0$ be positive real numbers such that the following hold:
	\begin{enumerate}
		\item Each of the points $P_1,P_2,\dots P_{d+1}$ lies on at most $n_0 q^d$ lines that meet $M$.
		\item $|M|\ge d_0q^{d+1}$.
	\end{enumerate}
	Then there exists a $(d+1)$-dimensional subspace $U$ on $\tau$ with
	\begin{align}\label{lowerboundonUcapMaddedbyKlaus}
		|U\cap M|>(2q)^{d+1}\left(\frac{d_0}{4n_0}\right)^{2^{d+1}-1}.
	\end{align}
\end{Proposition}
\begin{proof}
	We prove the following more general result. For each $s\in\{0,\dots,d+1\}$, there exists a set $\V_s$ of $s$-dimensional subspaces satisfying $|\V_s|\ge (\frac12)^sd_0q^{d+1-s}$ such that each $V\in \V_s$ satisfies
	\begin{align}
		&\{P_i\mid 1\le i\le s\}\subseteq V&
		&\text{and}&
		&|V\cap M|\ge(2q)^s\left(\frac{d_0}{4n_0}\right)^{2^s-1}.\label{EQ_IndProp}
	\end{align}
	We use induction on $s$. For $s=0$ we take $\V_0=M$. For the induction step $s\to s+1$, we assume the existence of $\V_s$ with the desired properties. For $V\in \V_s$ we know from the induction hypothesis that (\ref{EQ_IndProp}) holds and, since $\tau\cap M=\emptyset$ by hypothesis of the present lemma, this also implies $V\not\le\tau$, so that $P_{s+1}\notin V$. Now the previous lemma, applied with $c=2^s\left(\frac{d_0}{4n_0}\right)^{2^s-1}$, $\V=\V_s$ and $\mu=n_0q^d$ and $\gamma=\frac12$, proves the existence of a set $\V_{s+1}=\W$ with the desired properties.

	For $s=d+1$ we find $|\V_{d+1}|>0$, so each element $U$ of $\V_{d+1}$ satisfies the claim of the lemma.
\end{proof}

As mentioned earlier, we end this section with two technical lemmas that will be needed several times in the next section.

\begin{Lemma}\phantomsection\label{LemmaWithBounds}
Let $n,k,c\ge 1$ and $q\ge 2$ be integers. 
	\begin{enumerate}[label=(\alph*)]
		\item\label{boundsgauss_simple}
			If $n>k>0$ and $q\ge 4$ then
			\begin{align*}
				(q+1)q^{k(n-k)-1}\le\Gauss{n}{k}_q\le(q+2)q^{k(n-k)-1}.
			\end{align*}
		\item\label{boundsgauss_help}
			If $q>c^2+c$ then
			\begin{align*}
				(q^2+q+2)^c\le (q+c+1)q^{2c-1}.
			\end{align*}
		\item\label{boundsgauss_potenz}
			If $q>c^2+c$ we have $\gauss{c+1}{1}^c\le (q+c+1)q^{c^2-1}$.
	\end{enumerate}
\end{Lemma}
\begin{proof}
	\begin{enumerate}[label=(\alph*)]
		\item
			The lower bound follows from $0<k<n$ and for the upper bound we refer to \cite[Lemma 34]{Ferdinand&Klaus}.
		\item
			This can be checked by hand for $c=1$ and $c=2$, so we assume that $c\ge 3$. By expansion we see that $(q^2+q+2)^c=\sum_{i=0}^{2c}a_iq^i$ where
			\begin{align*}
				a_{2c-i}&=\sum_{j=0}^{\lfloor i/2\rfloor}\binom{c}{j}\binom{c-j}{i-2j}2^j,
			\end{align*}
			since a term $q^{2c-i}$ occurs in the expansion, if for some $j$ with $2j\le i$ we first choose the number 2 from $j$ terms $q^2+q+2$, secondly we choose the number $q$ from $i-2j$ of the remaining $c-j$ terms $q^2+q+2$ and finally we choose the number $q^2$ from the remaining terms $q^2+q+2$.
			
			Now, we claim $a_{2c-i}\le c^i$ for all $i$. Using $c\ge 3$, this can be verified for $i\le 5$ by straightforward calculation. Thus, suppose that $i\ge 6$. Then
			\begin{align}\label{bounda2cminusi}
				a_{2c-i}=\sum_{j=0}^{\lfloor i/2\rfloor}\frac{c!\cdot 2^j}{(c+j-i)!(i-2j)!j!}\le c^i\sum_{j=0}^{\lfloor i/2\rfloor}\underbrace{\frac{2^j}{c^j(i-2j)!j!}}_{=:b_{ij}}.
			\end{align}
			We next show $b_{ij}\le\frac2{i+2}$ for admissible $i,j$, that is, for $i,j$ with $2j\le i\le 2c$ and $i\ge 6$. Using $i\ge 6$, this follows from direct calculation if $j\le 3$. Otherwise $j\ge 4$ and $i\ge 8$, so $j!\ge 2^j$ and hence $b_{ij}\le c^{-j}\le \frac2{i+2}$, since $i\le 2c$. Thus we have established the bound for $b_{ij}$ and using it in (\ref{bounda2cminusi}) we find $a_{2c-i}\le c^i$ for $i\ge 6$. Hence $a_{2c-i}\le c^i$ for all $i\in\{0,\dots,2c\}$.
			
			It follows that
			\begin{align*}
			\sum_{i=0}^{2c-2}a_iq^i
			&\le \sum_{i=2}^{2c}c^iq^{2c-i}
			= q^{2c-2}c^2\sum_{i=0}^{2c-2}\frac{c^i}{q^i}
			\le \frac{q^{2c-2}c^2}{1-c/q}<q^{2c-1}
			\end{align*}
			where we have used $q>c^2+c$ in the last step. Since $a_{2c}=1$ and $a_{2c-1}=c$, this proves the claim in \ref{boundsgauss_help}.
		\item
			Since $\gauss{c+1}{1}_q\le (q^2+q+2)q^{c-2}$ this is a corollary to the previous claim.\qedhere
	\end{enumerate}
\end{proof}

\begin{Lemma}\label{newlemma}
	Let $U$ be a subspace of dimension $d+1$ of $\PG(2d,q)$, let $Y$ be a point of $U$ and let $T$ denote the set of all $(d-1)$-subspaces $\pi$ with $\pi\cap U=Y$.
	\begin{enumerate}[label=(\alph*)]
		\item\label{newlemma_onpoint}
			$|T|=q^{d^2-1}$.
		\item\label{newlemma_online}
			If $\ell$ is a line with $\ell\cap U=Y$, then $\ell$ lies on $q^{d^2-d-2}$ subspaces of $T$.
		\item\label{newlemma_onpointmeetsline}
			If $\ell$ is a line such that $\ell\cap U$ is a point different from $Y$, then $q^{d^2-d-1}$ subspaces of $T$ meet $\ell$ in a point.
		\item\label{newlemma_inhyperplane}
			If $H$ is a hyperplane of PG(2d,q) with $Y\in H$, then $H$ contains no subspace of $T$ if $H$ contains $U$ and it contains $q^{d^2-d}$ subspaces of $T$ if $H$ does not contain $U$.
	\end{enumerate}
\end{Lemma}
\begin{proof}
	See Theorem 3.3 in Hirschfeld \cite{Hirschfeld}.
\end{proof}
\end{section}

\begin{section}{The chromatic number of the graph \texorpdfstring{$\Gamma_d(q)$}{Gamma\_d(q)}}

In this section we prove the main result of this paper. In particular, we show that, given a \emph{large} value of $q$ and an integer $d\ge3$ for which Conjecture \ref{conject} holds, the chromatic number of $\Gamma_d(q)$ is $\gauss{d+2}{1}_q-q$. From Remark \ref{labelnewbyklaus24} we have $\chi(\Gamma_d(q))\le\gauss{d+2}{1}_q-q$.

In fact we will prove the following more general result, which does not depend on Conjecture \ref{conject}.

\begin{Theorem} \label{hypo_on_dq}
	Let $d\ge 3$ and $\alpha\ge 5$ be integers and let $\F=\{F_i\mid i=1,\dots,\gauss{d+2}{1}_q-q\}$ be a multiset (so we allow $F_i=F_j$ for $i\not=j$) of $\gauss{d+2}{1}_q-q$ independent sets of flags of type $\{d-1,d\}$ in $\PG(2d,q)$, whose union consists of all flags of this type in $\PG(2d,q)$.
	We put $S:=\{1\le i\le\gauss{d+2}{1}_q-q:|F_i|>e_1\}$ with
	\begin{align*}
		e_1:=\alpha q^{d^2+d-2},
	\end{align*}
	and suppose the following:
	\begin{enumerate}[label=(\Roman*)]
		\item\label{I}
			$q>3\cdot 112^{2^{d+1}-1}\cdot2^{-d-1}$ and $q\ge\frac{3}{2}\alpha^2+\frac{21}{2}\alpha+17$.
		\item\label{II}
			For $i\in S$ the set $F_i$ is one of the independent sets defined in Example \ref{Examples_IndependentSets}, which implies $g_0\le|F_i|\le e_0$, where
            \begin{align*}
				g_0&:=\Gauss{2d}{d+1}_q\cdot\Gauss{d+1}{1}_q&
				&\text{and}&
				e_0&:=g_0+\Gauss{2d-1}{d-1}_qq^d.
			\end{align*}
		\item\label{III}
			For distinct $i,j\in S$ the independent sets $F_i$ and $F_j$ have distinct generic parts.
		\item\label{IIII}
			For at least $\frac12|S|$ indices $i\in S$ the generic part of $F_i$ is based on a point.
	\end{enumerate}
	Then, each $F\in\F$ is a set described in Example \ref{Examples_IndependentSets}.\ref{Example_One} (that is, it is based on a point $P_F$) and the points $P_F$, $F\in\F$, are $\gauss{d+2}{1}_q-q$ mutually distinct points of a subspace of dimension $d+1$.
\end{Theorem}

The proof of this theorem is carried out in Lemma \ref{boundonxi} through Theorem \ref{Czerothm}. In all these lemmas and results, we suppose that $\F$ is as in the theorem, we assume that $d\ge 3$, $\alpha\ge 5$ and that \ref{I}-\ref{IIII} are satisfied. Using \ref{II} we define
\begin{align*}
I:=\{i\in S\mid \text{$i\in S$,\  the generic part of $F_i$ is based on a point $P_i$}\}.
\end{align*}
We consider $q$ as fixed and define
\begin{align*}
\theta_j:=\gauss{j+1}{1}_q
\end{align*}
for all integers $j\ge 0$. We also omit the index $q$ in Gaussian coefficients $\gauss{j}{k}_q$ and thus write $\gauss{j}{k}$.

\begin{Remark}\label{newremark}
As $q>\alpha$, then $e_1<g_0<e_0$. Hence, every $F\in\F$ satisfies $|F|\le e_0$. More precisely, Hypothesis (II) shows that $g_0\le |F_i|\le e_0$ for $i\in S$, and $|F|\le e_1$ for the remaining sets $F$ of $\cal F$.
\end{Remark}

\begin{subsection}{Construction of the subspace \texorpdfstring{$U$}{U}}
In this section we will construct a subspace $U$ of dimension $d+1$ that contains the base point $P_i$ of $F_i$ for many indices $i\in I$, similar to Example \ref{ExampleCoverings}.

\begin{Lemma}\label{boundonxi}
	For every subset $\GG$ of $\F$ we have
	\begin{align*}
		\left|\bigcup_{F\in\GG}F\right|\ge|\GG|e_0-\left(q^2+\frac92q+10\right)q^{d^2+2d-3}.
	\end{align*}
\end{Lemma}
\begin{proof}
	Since $|F|\le e_0$ for every $F\in\F$, it is sufficient to prove the statement in the case when $\GG=\F$. Then $|\GG|=\theta_{d+1}-q$ and $|\bigcup_{F\in\GG}F|$ is equal to the number $\gauss{2d+1}{d+1}\theta_d$ of all flags of type $\{d-1,d\}$ in $\PG(d,q)$. Using this, a direct calculation proves the statement when $d=3$ or $d=4$. For $d\ge 5$ we use $|\GG|\le q\theta_d$ and find
	\begin{align*}
		|\GG|e_0-\left|\bigcup_{F\in\GG}F\right|
& 
\le
q\theta_de_0-\theta_d\Gauss{2d+1}{d+1}
\\
 &\mathrel{\text{\makebox[\widthof{\scriptsize\ref{LemmaWithBounds}}][c]{$\le$}}}\theta_d\left(\Gauss{2d}{d+1}q\theta_d +\Gauss{2d-1}{d}q^{d+1}-\Gauss{2d+1}{d+1}\right)\\
			 &\mathrel{\text{\makebox[\widthof{\scriptsize\ref{LemmaWithBounds}}][c]{$=$}}}\theta_d\Gauss{2d-1}{d}\left(\frac{q^{2d}-1}{q^{d+1}-1}q\theta_d+q^{d+1}-\frac{(q^{2d+1}-1)(q^d+1)}{q^{d+1}-1}\right)\\
			 &\mathrel{\text{\makebox[\widthof{\scriptsize\ref{LemmaWithBounds}}][c]{$\le$}}}\theta_d\Gauss{2d-1}{d}(q^2+q+2)q^{2d-3}\ \ \ \text{(use $d\ge 5$ and $q$ and 4.1 (I))}\\
			&\overset{\text{\scriptsize\ref{LemmaWithBounds}}}{\le}\theta_d(q+2)(q^2+q+2)q^{d(d+1)-4}.
	\end{align*}
	From Hypothesis \ref{I} of Theorem \ref{hypo_on_dq} we deduce $\theta_d\le (q+\frac32)q^{d-1}$. Using Hypothesis \ref{I} again we see that $(q+\frac{3}{2})(q+2)(q^2+q+2)\le q^2(q^2+\frac92q+10)$ and the statement follows.
\end{proof}

\begin{Lemma}\label{4-space_2}
	Let $U$ be a $(d+1)$-dimensional subspace. Denote by $c_1$ the number of indices $i\in I$ with $P_i\notin U$ and by $c_3$ the number of independent sets $F\in\F$ with $|F|\le e_1$. Then there is some $x\in\{c_1,|I|-c_1\}$ with $x+2c_3\le 2(q+4+\alpha)q^{d-1}$.
\end{Lemma}
\begin{proof}
	From Hypothesis \ref{IIII} in Theorem \ref{hypo_on_dq} we know that $|I|\ge\frac12(|\F|-c_3)$ and we define $J:=\{i\in I:P_i\notin U\}$. Then, for all $j\in J$ and all $i\in I\setminus J$ the generic parts of the sets $F_i$ and $F_j$ share the flags $(\pi,\tau)$ with $P_i=\pi\cap U$ and $P_j\in \pi$ and Lemma \ref{newlemma} (b) implies that there are $\theta_{d}q^{d^2-d-2}$ such flags. For given $j\in J$ it is obvious that distinct $i$ in $I\setminus J$ yield distinct $\theta_{d}q^{d^2-d-2}$ such flags. Hence, for all $j\in J$ the set $F_j$ contains at least $|I\setminus J|\theta_{d}q^{d^2-d-2}$ flags that are contained in $F_i$ for some $i\in I\setminus J$. Using Remark \ref{newremark} it follows that
	\begin{align*}
		\left|\bigcup_{i\in I}F_i\right|\le|I|e_0 -\underbrace{|J||I\setminus J|}_{\mathclap{=c_1(|I|-c_1)}}\theta_dq^{d^2-d-2}
	\end{align*}
	and Lemma \ref{boundonxi} applied to the set $\GG:=\{F_i\mid i\in I\}\cup\{F\in\F:|F|\le e_1\}$ shows
	\begin{align}\label{eqnxsd1}
		c_3(e_0-e_1)+c_1(|I|-c_1)\theta_dq^{d^2-d-2}\le A:=\left(q^2+\frac92q+10\right)q^{d^2+2d-3}.
	\end{align}
	In particular, we already have $c_3(e_0-e_1)\le A$ and we set $B:=(q+4+\alpha)q^{d-1}$. From the definition of $e_0$ and $e_1$ we see that $e_0-e_1\ge ((q+1)^2-\alpha q)q^{d^2+d-3}$. Hypothesis \ref{I} of Theorem \ref{hypo_on_dq} implies that $B(e_0-e_1)>A$ and hence we have $c_3<B$. We now show that one of the numbers in $\{c_1,|I|-c_1\}$ is at most $2(B-c_3)$. Suppose that this is wrong. Then
	\begin{align}\label{eqnxsd2}
		&\left(c_1-2(B-c_3)\right)\left(|I|-c_1-2(B-c_3)\right)\geq 0 \nonumber \\
		\Leftrightarrow  \quad &c_1(|I|-c_1)\ge 2(B-c_3)(|I|-2B+2c_3).
	\end{align}
	Since $|I|\ge\frac12(|\F|-c_3)$, it follows from \eqref{eqnxsd1} and \eqref{eqnxsd2} that $f(c_3)\le A$ where $f$ is the polynomial in $x$ given by
	\begin{align*}
		f:=x(e_0-e_1)+2(B-x)\left(\frac12(|\F|-x)-2B+2x\right)\theta_dq^{d^2-d-2}.
	\end{align*}
	Since $f$ has degree two with negative coefficient in $x^2$ and since $0\le c_3<B$, we have $\min\{f(0),f(B)\}\le f(c_3)$ so $f(0)\le A$ or $f(B)\le A$. But $f(B)=B(e_0-e_1)$ and we have already seen that this is larger than $A$. Hence $f(0)\le A$, that is,
	\begin{align*}
		2B\left(\frac12|\F|-2B\right)\theta_dq^{d^2-d-2}\le A.
	\end{align*}
	Using $|\F|=\theta_{d+1}-q\ge (q+1)q^d$ and $\theta_d\ge (q+1)q^{d-1}$, it follows that
	\begin{align*}
		(q+1)B((q+1)q^d-4B)\le\left(q^2+\frac92q+10\right)q^{2d}.
	\end{align*}
	Using the definition of $B$, this gives
	\begin{align}\label{eqnjhuz3}
		(q+1)(q+4+\alpha)(q^2-3q-16-4\alpha)\le\left(q^2+\frac92q+10\right)q^2.
	\end{align}
	From the hypothesis of Theorem \ref{hypo_on_dq} we have $5\le \alpha\le q$. This implies that \eqref{eqnjhuz3} must also be satisfied when $\alpha$ is replaced by $5$ or by $q$, but this contradicts the lower bound in Hypothesis \ref{I} of Theorem \ref{hypo_on_dq} for $q$.
\end{proof}

\begin{Lemma}\label{4-space_3}
	There exists a (d+1)-dimensional subspace $U$ such that
	\begin{align*}
		|\{i\in I\mid P_i\notin U\}|+2\cdot |\{F\in\F:|F|<g_0\}|\le 2(q+4+\alpha)q^{d-1}.
	\end{align*}
	\end{Lemma}
\begin{proof}
	Let $c_3$ be the number of $F\in\F$ with $|F|<g_0$ and thus $|F|\le e_1$, see Remark \ref{newremark}. Let $\beta$ be the number of $F\in\F$ with $|F|\ge g_0$ whose generic parts are based on points. Let $G_1,\dots,G_\beta$ denote these independent sets and let $R_1,\dots,R_\beta$ be the respective base points of their generic parts. As $|\F|=\theta_{d+1}-q$, Hypothesis (IV) of Theorem \ref{hypo_on_dq} shows that $\beta\ge\frac12(\theta_{d+1}-q-c_3)$. For all $i\in\{1,\dots,\beta\}$ define
	\begin{align*}
		g_i:=\left|G_i\cap\bigcup_{j=1}^{i-1} G_j\right|.
	\end{align*}
	From Remark \ref{newremark}, we have $|\bigcup_{j=1}^{i}G_j|\le i e_0-\sum_{j=1}^{i}g_i$ for all $i\le\beta$. We may assume that the sequence $g_1,\dots,g_\beta$ is monotone increasing. We define $j:=\lceil\frac14 q^{d+1}\rceil+\theta_d+\theta_{d-1}-d$ and claim that $g_j<5q^{d^2-2}\theta_d$.

Assume this is not true. Then $\sum_{i=j}^\beta g_i\ge (\beta-j+1)5 q^{d^2-2}\theta_d$ and Lemma \ref{boundonxi} applied to the set $\{F\in\F:|F|\le e_1\}\cup\{G_i\mid j\le i\le \beta\}$ implies
	\begin{align*}
		(\beta-j+1)5 q^{d^2-2} \theta_d +c_3(e_0-e_1)\le\left(q^2+\frac{9}{2}q+10\right)q^{d^2+2d-3}.
	\end{align*}
	On the left hand side we use the definition of $j$ and use $\beta\ge\frac12(\theta_{d+1}-q-c_3)$. In the resulting expression the coefficient of $c_3$ is $e_0-e_1-\frac{5}{2}q^{d^2-2}\theta_d$ and in view of Lemma \ref{LemmaWithBounds} and Hypothesis \ref{I} of Theorem \ref{hypo_on_dq} this is positive. Hence, we may omit the term with $c_3$ and find 
	\begin{align*}
5q^{d^2-2}\theta_d\left(\left\lfloor\frac{1}{4}q^{d+1}\right\rfloor-\frac{1}{2}(\theta_d+q)-\theta_{d-1}+d\right)\le\left(q^2+\frac92q+10\right)q^{d^2+2d-3}.
	\end{align*}
	However, this contradicts Hypothesis \ref{I} of Theorem \ref{hypo_on_dq}.
	
Hence $g_j<5q^{d^2-2}\theta_d$. Now, let $Q_1,\dots,Q_{d+1}\in\{R_{j-\theta_{d-1}},\dots,R_j\}$ be such that $\tau:=\langle Q_1,\dots,Q_{d+1}\rangle$ is a $d$-dimensional subspace and set
	\begin{align*}
		\R:=\{R_i:i\in\{1,\dots,j-\theta_{d-1}-1\}\text{ and } R_i\notin\tau\}.
	\end{align*}
	Then $|\R|\ge j-\theta_{d-1}-1-(|\tau|-d-1)=\lceil\frac{1}{4}q^{d+1}\rceil$.
	
	In the next step we show for all $i\in\{1,\dots,d+1\}$ that the point $Q_i$ lies on fewer than $7q^d$ lines that meet $\R$. Assume that this is false and let $i\in\{1,\dots,d+1\}$ be such that $Q_i$ lies on at least $7q^d$ lines that meet $\R$. Each of these lines lies in $\gauss{2d-1}{d-2}$ subspaces of dimension $d-1$ and two of these lines occur together in $\gauss{2d-2}{d-3}$ such subspaces. Hence there exist at least
	\begin{align*}
		7q^d\left(\Gauss{2d-1}{d-2}-7q^d\Gauss{2d-2}{d-3}\right)
			&=7q^d\left(\frac{q^{2d-1}-1}{q^{d-2}-1}-7q^d\right)\Gauss{2d-2}{d-3}\\
			&\ge 7q^{d^2-3}(q-7)=:z
	\end{align*}
	$(d-1)$-dimensional subspaces that contain one of the $7q^d$ lines. This shows that there exist $z\theta_d$ flags $(E,S)$ of type $\{d-1,d\}$ with $Q_i\in E$ and such that $E$ contains a point of $\R$. Since $Q_i=R_k$ for some $k$ with $j-\theta_{d-1}\le k\le j$, this implies that $z\theta_d\le g_k\le g_j<5q^{d^2-2}\theta_d$, which contradicts Hypothesis (I) of Theorem \ref{hypo_on_dq}. Consequently, for all $i\in\{1,\dots,d+1\}$ the point $Q_i$ lies in fact on fewer than $7q^d$ lines that meet $\R$.

	Finally, we apply Proposition \ref{4-space_1} with $d_0=|\R|/q^{d+1}\ge \frac{1}{4}$, $n_0=7$ and $M:=\R$ to find a $(d+1)$-dimensional subspace $U$ satisfying (\ref{lowerboundonUcapMaddedbyKlaus}). Using the lower bounds for $q$ of Hypothesis \ref{I} of Theorem \ref{hypo_on_dq} we conclude that $|U\cap R|\ge 3q^d>2(q+4+\alpha)q^{d-1}$. The statement of the lemma follows now from Lemma \ref{4-space_2}.
\end{proof}
\end{subsection}

\begin{subsection}{The proof of the theorem}
\begin{Notation}
	From now on we let $U$ be the (d+1)-dimensional subspace provided by Lemma \ref{4-space_3} and define the following sets:
	\begin{enumerate}[label=$\bullet$]
		 \item $C_0:=\{F_i\in\F\mid i\in I,\ P_i\in U\}$, $c_0:=|C_0|$,
		 \item $C_1:=\{F_i\in\F\mid i\in I,\ P_i\notin U\}$, $c_1:=|C_1|$,
		 \item $C_2:=\{F_i\in\F\mid i\notin{I}, |F_i|\ge g_0\}$, $c_2:=|C_2|$,
		 \item $C_3:=\{F_i\in\F\mid i\notin{I}, |F_i|<g_0\}$, $c_3:=|C_3|$,
		 \item $W:=\{P\in U\mid P\not=P_i\; \forall i\in I\}$,
		 \item $M:=\{(\pi,\tau)\in\bigcup_{F\in \F}F\mid\pi\cap U\ \text{is a point and }\pi\cap U\in W\}$.
	\end{enumerate}
\end{Notation}

\begin{Lemma}\label{ci_properties}
	With this notation the following hold:
	\begin{enumerate}[label=(\alph*)]
		\item\label{ci_properties_SmallSets}
			For all $F\in C_3$ we have $|F|\le e_1<g_0$.
		\item\label{ci_properties_partion}
			$C_0\cup C_1\cup C_2\cup C_3$ is a partition of $\F$.
		\item\label{ci_properties_boundc1c3}
			$c_1+2c_3\le{{2(q+4+\alpha)q^{d-1}}}$.
		\item\label{ci_properties_sizeofW}
			$|W|=\theta_{d+1}-c_0\ge q$.
		\item\label{ci_properties_flagsofWonpointofM}
			For all $P\in W$ there are exactly $q^{d^2-1}\theta_d$ flags $(\pi,\tau)$ with $\pi\cap U=P$.
		\item\label{ci_properties_sizeofM}
			$|M|=|W|q^{d^2-1}\theta_d$.
		\item\label{ci_properties_IisLarge}
		$|I|=c_0+c_1\ge\frac12(\theta_{d+1}-q-c_3)$.
	\end{enumerate}
\end{Lemma}
\begin{proof}
	The first claim is implied by the choice of $\F$ with the properties given in Theorem \ref{hypo_on_dq}, especially Hypothesis \ref{II}. Claim \ref{ci_properties_partion} is obvious from the choice of $C_0$, $C_1$, $C_2$ and $C_3$. The choice of $U$ implies \ref{ci_properties_boundc1c3}. From Hypothesis \ref{III} in Theorem \ref{hypo_on_dq}, we know that the base points $P_i$ of the sets $F_i$ with $i\in I$ are pairwise distinct. Therefore we have $|W|=|U\setminus C_0|=|U|-|C_0|=\theta_{d+1}-c_0$. Since $c_0\le|\F|\le \theta_{d+1}-q$, this gives \ref{ci_properties_sizeofW}. Furthermore, from Lemma \ref{newlemma}~\ref{newlemma_onpoint} we know that each point $P\in W$ lies on $q^{d^2-1}$ subspaces of dimension $(d-1)$ that meet $U$ only in $P$ and each such subspace lies in $\theta_d$ subspaces of dimension $d$. Hence, for every point $P$ in $W$ exactly $q^{d^2-1}\theta_d$ flags $(\pi,\tau)$ of $M$ satisfy $\pi\cap U=P$, which proves \ref{ci_properties_flagsofWonpointofM} and \ref{ci_properties_sizeofM}. To see \ref{ci_properties_IisLarge} we first note that our definitions imply $|I|=c_0+c_1$ and that exactly $|\F|-c_3=\theta_{d+1}-q-c_3$ elements of $\F$ have at least $g_0$ elements. Then we recall from Hypothesis \ref{II} and \ref{IIII} in Theorem \ref{hypo_on_dq}, that every element of $\F$ with at least $g_0$ elements is based on a point or a hyperplane and that at least half of these are based on a point.
\end{proof}

\begin{Notation}\label{notation}
	Recall from Lemma \ref{Sizeofindependentsets} that the special parts of all independent sets given by Example \ref{Examples_IndependentSets} --- in particular of all independent sets $F\in\F$ with $|F|\ge g_0$ --- have cardinality at most $\Delta$, where
	\begin{align}\label{boundsDelta}
		\Delta:=\gauss{2d-1}{d}q^d\overset{{{\text{\scriptsize\ref{LemmaWithBounds}}}}}{\le}(q+2)q^{d^2-1}.
	\end{align}
\end{Notation}

\begin{Lemma}\phantomsection\label{HowCimeetsM}
	\begin{enumerate}[label=(\alph*)]
		\item\label{HowC0meetsM}
			If $F\in C_0$, then the generic part of $F$ does not contain a flag of $M$. In particular $|F\cap M|\le\Delta$.
		\item\label{HowC1meetsM}
			If $F\in C_1$, then $|F\cap M|\le |W|q^{d^2-d-2}\theta_d+\Delta$.
		\item\label{HowC2meetsM}
			If $F\in C_2$, then $F$ is based on a hyperplane $H$ and we have
			\begin{align*}
				|F\cap M|\le\begin{cases}
					\Delta &\text{if }U\le H,\\
					\Delta+|H\cap W|q^{d^2-d}\theta_{d-1} &\text{otherwise}.
				\end{cases}
		\end{align*}
	\end{enumerate}
\end{Lemma}
\begin{proof}
	\begin{enumerate}[label=(\alph*)]
		\item
			For all flags $(\pi,\tau)$ of the generic part of $F$ we have $\dim(\pi\cap U)\ge1$ or $\pi\cap U$ is the base point of $F$. Since $M$ only contains flags $(\pi,\tau)$ such that $\pi$ meets $U$ in a point that is not a base point of the generic part of some $F\in C_0$, this implies that these flags do not belong to $M$. Therefore $|F\cap M|\le\Delta$ follows from \ref{notation}.
		\item
			As $F\in C_1$ it is based on a point $P$ with $P\notin U$. If $Y\in W$, then according to Lemma \ref{newlemma}~\ref{newlemma_online} the point $P$ lies on exactly $q^{d^2-d-2}$ subspaces $\pi$ of dimension $d-1$ satisfying $\pi\cap U=Y$. Each of these lies on $\theta_d$ subspaces of dimension $d$. Hence, the generic part of $F$ contains exactly $|W|q^{d^2-d-2}\theta_d$ flags of $M$. Furthermore, the special part of $F$ contains at most $\Delta$ flags and thus at most this many flags of $M$.
		\item
			Since $F$ is not based on a point and has cardinality at least $g_0$, Hypothesis \ref{II} of Theorem \ref{hypo_on_dq} shows that $F$ is based on a hyperplane $H$. The generic part of $F$ consists of all flags $(\pi,\tau)$ of type $\{d-1,d\}$ with $\tau\le H$ and thus also $\pi\le H$. If $Y\in H\cap W$, then according to Lemma \ref{newlemma}~\ref{newlemma_inhyperplane} the number of $(d-1)$-subspaces $\pi$ of $H$ with $\pi\cap U=Y$ is zero for $U\le H$ and it is $q^{d^2-d}$ for $U\not\le H$. Since for every $(d-1)$-subspace of $H$ the number of $d$-subspaces of $H$ containing it is $\theta_{d-1}$, it follows that the generic part of $F$ contains no flag of $F$ for $U\le H$ and exactly $|H\cap W|q^{d^2-d}\theta_{d-1}$ flags of $M$ for $U\nleq H$. Finally, since the special part of $F$ contains at most $\Delta$ flags, this implies the claim.\qedhere
		\end{enumerate}
\end{proof}

\begin{Lemma}\label{Wglobal}
	Suppose that $z$ is an integer such that there is at most one hyperplane of $U$ which contains more than $z$ points of $W$. Then $M$ has size at most
	\begin{align*}
		 (c_0+c_1+c_2)\Delta+c_1|W|q^{d^2-d-2}\theta_d+c_2zq^{d^2-d}\theta_{d-1}+c_3e_1+q^{d^2-1}\theta_{d-1}\theta_d.
	\end{align*}
\end{Lemma}
\begin{proof}
	Since every flag of $M$ is covered by some $F\in\F$, we have, using \ref{ci_properties} (f)
	\begin{align*}
	|W|q^{d^2-1}\theta_d\le \sum_{i=0}^3\left|\bigcup_{F\in C_i}F\cap M\right|.
	\end{align*}
	Now, if there exists a hyperplane of $U$ with more than $z$ points in $W$, then let $z'$ denote the number of its points in $W$ and otherwise put $z':=z$. Since every hyperplane of $U$ lies in $q^{d-1}$ hyperplanes of $\PG(2d,q)$ which do not contain $U$, Lemma \ref{HowCimeetsM}~\ref{HowC2meetsM} shows
	\begin{align*}
	\left|\bigcup_{{F\in C_2}}F\cap M\right|
	&\le\left(c_2-q^{d-1}\right)(\Delta+zq^{d^2-d}\theta_{d-1})+q^{d-1}(\Delta+z'q^{d^2-d}\theta_{d-1})\\
	&=c_2(\Delta+zq^{d^2-d}\theta_{d-1})+\underbrace{(z'-z)}_{\le\theta_d}q^{d^2-1}\theta_{d-1}.
	\end{align*}
	Finally, since $|F|\le e_1$ for $F\in C_3$ via Lemma \ref{ci_properties}~\ref{ci_properties_SmallSets}, the assertion follows from Lemma \ref{HowCimeetsM}~\ref{HowC0meetsM} and \ref{HowC1meetsM}.
\end{proof}

\begin{Lemma}\label{twoplanes}
	Let $\tau_1$ and $\tau_2$ be distinct hyperplanes of $U$ and set $W':=(\tau_1\cup\tau_2)\cap W$. Then
	\begin{align*}
		|W'|\theta_{d-1}\le(c_1+c_2)q^{d-3}(2q+7)+q^{2d-3}((\alpha+3)q+\alpha^2+4\alpha).
	\end{align*}
\end{Lemma}
\begin{proof}
	We have $|W'|\le |(\tau_1\cup\tau_2)\cap U|=q^d+\theta_d$. We set $$M':=\{(\pi,\tau)\in M:\pi\cap U\in W'\},$$
	that is, $M'$ consists of all flags $(\pi,\tau)$ of type $\{d-1,d\}$ such that $\pi\cap U$ is a point that lies in $W'$. Lemma \ref{ci_properties}~\ref{ci_properties_flagsofWonpointofM} shows that $|M'|=|W'|q^{d^2-1}\theta_d$. Each flag of $M'$ lies in at least one of the independent sets of $\F=C_0\cup C_1\cup C_2\cup C_3$. Hence
	\begin{align}
		|W'|q^{d^2}\theta_{d-1}\le|W'|q^{d^2-1}\theta_d=|M'|\le d_0+d_1+d_2+d_3,\label{maineqintwoplanes}
	\end{align}
	where for all $i\in\{1,\dots,4\}$ we let $d_i$ denote the number of elements of $M'$ that lie in some member of $C_i$. Now, we determine upper bounds on these numbers $d_0$, $d_1$, $d_2$ and $d_3$ in 4 steps.

	First, we consider an independent set $F\in C_0$. Then $|F| \geq g_0$ and $F$ is based on a point $P\in U$. We know from Lemma \ref{HowCimeetsM} that only the special part $T$ of $F$ may contribute to $M'$. Therefore, we study $T$ and the three possible structures that $T$ may have. Note that we frequently make use of Lemma \ref{newlemma} without explicit mention.
	\begin{enumerate}[label=$\bullet$]
	\item
		First, assume that there is a line $l$ with $P\in l$ such that $T$ consists of all flags $(\pi,\tau)$ of type $\{d-1,d\}$ with $l\le\tau$ and $P\notin\pi$. Recall from Example \ref{Examples_IndependentSets} that every flag $(\pi,\tau)\in T$ is determined by $\pi$, since $\tau=\langle P,\pi\rangle$. Then we have
		\begin{align*}
			|T\cap M'|
				&=\begin{cases}
					|l\cap W'|q^{d^2-1} &\text{if }l\le U,\\
					|W'|q^{d^2-d-1} &\text{if }l\cap U=P
				\end{cases}
			\intertext{and, using $|W'|\le q^d+\theta_d$ as well as the fact that $|l\cap W'|$ is at most $q$ for $P\in\tau_1\cup\tau_2$ and at most $2$ otherwise, we have}
			|T\cap M'|
				&\le\begin{cases}
					q^{d^2}							&\text{if }P\in\tau_1\cup\tau_2,\\
					(q^d+\theta_d)q^{d^2-d-1}		&\text{otherwise.}
			\end{cases}
		\end{align*}
	\item
		Secondly, assume that there is a hyperplane $H$ with $P\in H$ such that $T$ consists of all flags $(\pi,\tau)$ of type $\{d-1,d\}$ with $P\in\tau\le H$ and $P\notin\pi$. As before, every flag $(\pi,\tau)\in T$ is determined by $\pi$, since $\tau=\langle P,\pi\rangle$. If $U\le H$, then Lemma \ref{newlemma}~\ref{newlemma_inhyperplane} implies $T\cap M'=\emptyset$ and otherwise it implies
		\begin{align*}
			|T\cap M'|
			&=|H\cap W'|q^{d^2-d}\\
			&\le\begin{cases}
				|W'\cap\tau_i|q^{d^2-d} &\text{if }H\cap U=\tau_i\text{ for some }i\in\{1,2\},\\
				(q^{d-1}+\theta_{d-1})q^{d^2-d} &\text{otherwise.}
			\end{cases}
		\end{align*}
		Notice that $H\cap U=\tau_i$ for some $i\in\{1,2\}$ implies $P\in\tau_i$ and thus
		$|W'\cap\tau_i|\le\theta_d-1=q\theta_{d-1}$. Therefore, we have
		\begin{align*}
			|T\cap M'|\le\begin{cases}
				\theta_{d-1}q^{d^2-d+1} &\text{if }P\in\tau_1\cup\tau_2,\\
				(q^{d-1}+\theta_{d-1})q^{d^2-d} &\text{otherwise.}
			\end{cases}
		\end{align*}
	\item
		Finally, suppose that the special part $T$ is not based on a line or a hyperplane. Then Lemma \ref{Sizeofindependentsets} shows
		\begin{align*}
			|T\cap M'|
			&\le |T|\le q^{d}(1+q^{-1})\theta_{d-2}\theta_{d-1}^{d-1}\overset{\text{\scriptsize\ref{LemmaWithBounds}~\ref{boundsgauss_potenz}}}{\le}\theta_1(q+d)\theta_{d-2}q^{d^2-d-1}.
		\end{align*}
	\end{enumerate}
	Using Hypothesis \ref{I} of Theorem \ref{hypo_on_dq} we may summarize these three upper bounds into
	\begin{align*}
		|T\cap M'|\le
		\begin{cases}
			\theta_{d-1}q^{d^2-d+1}				&\text{for }P\in\tau_1\cup\tau_2,\\
			(q^d+\theta_{d})q^{d^2-d-1}			&\text{otherwise.}
		\end{cases}
	\end{align*}
	Note that the bound given for $P\in\tau_1\cup\tau_2$ is a weaker bound than the bound for $P\notin\tau_1\cup\tau_2$. Now, since different sets $F\in C_0$ are based on different points $P$ (see Hypothesis \ref{III} in Theorem \ref{hypo_on_dq}) and since $\tau_1\cup\tau_2$ contains $q^d+\theta_d$ points, we find
	\begin{align*}
		d_0\le c_0(q^d+\theta_d)q^{d^2-d-1}+(q^d+\theta_d)\theta_{d-1}q^{d^2-d+1}\le12q^{d^2+d},
	\end{align*}
	where the last step uses the trivial bounds $c_0\le\theta_{d+1}\le2q^{d+1}$, $\theta_d\le 2q^d$ and $\theta_{d-1}\le 2q^{d-1}$.

	Secondly, for $F\in C_1$ we see that $F$ contains at most $|W'|q^{d^2-d-2}\theta_d+\Delta$ flags of $M'$ analogously to \ref{HowCimeetsM}~\ref{HowC1meetsM}, which already proves $d_1\le c_1(|W'|q^{d^2-d-2}\theta_d+\Delta)$. Now, we use $\Delta\le(q+2)q^{d^2-1}$ given in Inequality (\ref{boundsDelta}) as well as $|W'|\le q^d+\theta_d$ and have
	\begin{align*}
		 d_1&\mathrel{\text{\makebox[\widthof{\scriptsize\ref{LemmaWithBounds}~\ref{boundsgauss_simple}}][c]{$\le$}}}c_1q^{d^2-d-2}((q^d+\theta_d)\theta_d+q^{d+2}+2q^{d+1})
	\intertext{For $d=3$ simple calculations show that this is smaller than $c_1q^{d^2+d-3}(2q+7)$ and for $d\ge4$ we receive the same upper bound via}
			 d_1 &\overset{\text{\scriptsize\ref{LemmaWithBounds}~\ref{boundsgauss_simple}}}{\le}c_1q^{d^2-1}(2q^{d-1}+6q^{d-2}+4q^{d-3}+q+2)\overset{{\text{\scriptsize\ref{hypo_on_dq}~\ref{I}}}}{\le}c_1q^{d^2+d-3}(2q+7).
	\end{align*}

	Thirdly, we consider $F\in C_2$. Then $|F|\ge g_0$ and $F$ is based on a hyperplane $H$. If $U\subseteq H$ and $(\pi,\tau)$ is a flag of $F$, then $\dim(\pi\cap U)\ge1$ and thus $F\cap M'=\emptyset$. Therefore, we only need to study the case $U\not\le H$, which implies $\dim(U\cap H)=d$. Then, analogously to the proof of \ref{HowCimeetsM}~\ref{HowC2meetsM}, we see that the number of flags of $M'$ in the generic part of $F$ is $|H\cap W'|q^{d^2-d}\theta_{d-1}$ and we have
	\begin{align*}
		|H\cap W'|q^{d^2-d}\theta_{d-1}\le\begin{cases}
			|W'\cap \tau_i|q^{d^2-d}\theta_{d-1}	&\text{if }H\cap U=\tau_i\text{ for }i\in\{1,2\},\\
			(q^{d-1}+\theta_{d-1})q^{d^2-d}\theta_{d-1}				&\text{otherwise.}
		\end{cases}
	\end{align*}
	Since there are exactly $q^{d-1}$ hyperplanes that meet $U$ in $\tau_1$ and as many that meet $U$ in $\tau_2$, it follows that the number of flags of $M'$ that lie in the generic part of at least one independent set of $C_2$ is at most
	\begin{align*}
		c_2(q^{d-1}+\theta_{d-1})q^{d^2-d}\theta_{d-1}+q^{d^2-1}\left(|W'\cap \tau_1|+|W'\cap\tau_2|\right)\theta_{d-1}.
	\end{align*}
	The special part of each independent set of $C_2$ has $\Delta$ flags and thus at most this many flags of $M'$. Using
	\begin{align*}
		|W'\cap \tau_1|+|W'\cap \tau_2|
			 &\le|\tau_1|+|\tau_2|=2\theta_d\overset{\text{\scriptsize\ref{LemmaWithBounds}~\ref{boundsgauss_simple}}}{\le}2(\theta_1+1)q^{d-1}
	\end{align*}
	it follows that
	\begin{align*}
		d_2\le c_2\Delta+c_2(q^{d-1}+\theta_{d-1})q^{d^2-d}\theta_{d-1}+2q^{d^2+d-2}(\theta_1+1)\theta_{d-1}.
	\end{align*}
	We now show that this bound implies
	\begin{align}
		d_2\le c_2q^{d^2+d-3}(2q+7)+2q^{d^2+2d-4}(q^2+4q+4)\label{neweqford2}.
	\end{align}
	For $d=3$, this can easily be verified for all $q\ge 3$. For $d\geq 4$, we use $\Delta\le(q+2)q^{d^2-1}$ given in Inequality (\ref{boundsDelta}) as well as the upper bound given in Lemma \ref{LemmaWithBounds}~\ref{boundsgauss_simple} to find
	\begin{align*}
		d_2\le c_2q^{d^2-1}(2q^{d-1}+6q^{d-2}+4q^{d-3}+q+2)+2q^{d^2+2d-4}(\theta_1+1)^2
	\end{align*}
	and \ref{hypo_on_dq}~\ref{I} implies Equation (\ref{neweqford2}).

	Finally, we note that for $F\in C_3$ we trivially have $|F\cap M'|\le |F|\le e_1$ and, using $c_3\le(q+4+\alpha)q^{d-1}$ from Lemma \ref{ci_properties}~\ref{ci_properties_boundc1c3} as well as $e_1=\alpha q^{d^2+d-2}$, this shows
	\begin{align*}
		d_3\le c_3e_1\le(\alpha q+\alpha^2+4\alpha)q^{d^2+2d-3}.
	\end{align*}

	We have now proved upper bounds of $d_0$, $d_1$, $d_2$ and $d_3$. Substituting these upper bounds in Equation (\ref{maineqintwoplanes}) and dividing by $q^{d^2-d-1}$ yields
	\begin{multline*}
		|W'|q^{d+1}\theta_{d-1}\le(c_1+c_2)q^{2d-2}(2q+7)+(\alpha q+\alpha^2+4\alpha)q^{3d-2}\\
			+q^{2d}(2q^{d-1}+8q^{d-2}+8q^{d-3}+12q)
	\end{multline*}
	and, using the lower bound on $q$ given in Hypothesis \ref{I} of Theorem \ref{hypo_on_dq}, this implies the claim.
\end{proof}

\begin{Lemma}\label{Watmostquadraticinq}
	We have $|W|\le(\alpha+3)q^{d-1}$.
\end{Lemma}
\begin{proof}
	Let $\tau_1$ and $\tau_2$ be hyperplanes of $U$ such that $|\tau_1\cap W|\ge|\tau_2\cap W|\ge|\tau\cap W|$ for every hyperplane $\tau$ of $U$ other than $\tau_1$ and set $z:=|\tau_2\cap W|$. Then, Lemmas \ref{Wglobal} and \ref{ci_properties}~\ref{ci_properties_sizeofM} show
	\begin{align*}
		|W|(q^{d+1}-c_1)q^{d^2-d-2}\theta_d
			&\le(c_0+c_1+c_2)\Delta+c_2zq^{d^2-d}\theta_{d-1}+c_3e_1\\
			&\hphantom{\le{}}\mathrel{+}q^{d^2-1}\theta_{d-1}\theta_d
	\end{align*}
	and, if we set $\delta:=c_1+c_2+c_3$ and use $c_0+c_1+c_2+c_3=|\F|=\theta_{d+1}-q$ as well as $|W|=\theta_{d+1}-c_0=\delta+q$ from Lemma \ref{ci_properties}~\ref{ci_properties_sizeofW}, then this is equivalent to
	\begin{align*}
		0
			&\le(\theta_{d+1}-q)\Delta+q^{d^2-1}\theta_{d-1}\theta_d+c_3(e_1-\Delta)\\
			&\hphantom{\le{}}\mathrel{+}c_2q^{d^2-d}z\theta_{d-1}+(\delta+q)(c_1-q^{d+1})q^{d^2-d-2}\theta_d.
	\end{align*}
	Before we proceed, we simplify this inequality:
	\begin{enumerate}[label=$\bullet$,noitemsep]
		\item
			in the first term, since $\Delta$ is positive, we may replace $(\theta_{d+1}-q)$ by its upper bound $(q+2)q^{d}$ given in Lemma \ref{LemmaWithBounds}~\ref{boundsgauss_simple};
		\item
			in the second term we use $q^{d^2-1}\theta_{d-1}\theta_d\le(q+5)q^{d^2+2d-3}$, which follows from Lemma \ref{LemmaWithBounds}~\ref{boundsgauss_simple} and the lower bound on $q$ from Hypothesis \ref{I} of Theorem \ref{hypo_on_dq};
		\item
			in the third term, since the coefficient $e_1-\Delta$ of $c_3$ is positive (this is implied by our assumption $\alpha\ge 5$), we may replace $c_3$ by its upper bound $(q+4+\alpha)q^{d-1}$ given in Lemma \ref{ci_properties}~\ref{ci_properties_boundc1c3};
		\item
			and in the final term, since $c_1-q^{d+1}$ is negative (consider the upper bound $c_1\le2(q+4+\alpha)q^{d-1}$ given in Lemma \ref{ci_properties}~\ref{ci_properties_boundc1c3}), we may replace $(\delta+q)q^{d^2-d-2}\theta_d$ by its lower bound $\delta(q+1)q^{d^2-3}$ implied by Lemma \ref{LemmaWithBounds}~\ref{boundsgauss_simple}.
	\end{enumerate}
	This yields
	\begin{align}
		0
			&\le(q+2)q^d\Delta+(q+5)q^{d^2+2d-3}+(q+4+\alpha)q^{d-1}(e_1-\Delta)\nonumber\\
			 &\hphantom{\le{}}\mathrel{+}c_2q^{d^2-d}z\theta_{d-1}+\delta(q+1)q^{d^2-3}(c_1-q^{d+1}).\label{someequation}
	\end{align}
	Next we want to take care of the variable $z$ in the fourth term of this inequality. For that purpose we note that the preceding lemma is applicable to the set $W':=(\tau_1\cup\tau_2)\cap W$ and that $W'$ satisfies
	\begin{align*}
		|W'|\ge|\tau_1\cap W|+|\tau_2\cap W|-\theta_{d-1}\ge 2z-\theta_{d-1},
	\end{align*}
	that is, we have
	\begin{align*}
		2z\theta_{d-1}
			 &\mathrel{\text{\makebox[\widthof{\scriptsize\ref{hypo_on_dq}~\ref{I}}][c]{$\le$}}}|W'|\theta_{d-1}+\theta_{d-1}^2\overset{\text{\scriptsize\ref{LemmaWithBounds}~\ref{boundsgauss_simple}}}{\le}|W'|\theta_{d-1}+(q+2)^2q^{2d-4}\\
			&\overset{\text{\scriptsize\ref{hypo_on_dq}~\ref{I}}}{\le}|W'|\theta_{d-1}+2q^{2d-2}.
	\end{align*}
	We use the bound given by Lemma \ref{twoplanes} (where, for convenience, we replace the 7 by an 8) as well as $c_1+c_2\le\delta$ to first replace the first term on the right hand side and then simplify the result and have
	\begin{align}
		z\theta_{d-1}\le\delta q^{d-3}(q+4)+\frac{q^{2d-3}}{2}((\alpha+5)q+\alpha^2+4\alpha).\label{upperboundz}
	\end{align}
	Now, we reconsider Inequality (\ref{someequation}):
	\begin{enumerate}[label=$\bullet$,noitemsep]
		\item
			using Hypothesis \ref{I} of Theorem \ref{hypo_on_dq} we see that the coefficient $(q^2+q-4-\alpha)q^{d-1}$ of $\Delta$ therein is positive and thus we may replace $\Delta$ by its upper bound $(q+2)q^{d^2-1}$ given in Inequality (\ref{boundsDelta});
		\item
			the coefficients of $c_1$ and $c_2$ therein are non-negative and so we may substitute $c_1$ and $c_2$ by their respective upper bounds $2(q+4+\alpha)q^{d-1}$ and $\delta$, the first of which is given in Lemma \ref{ci_properties}~\ref{ci_properties_boundc1c3} and the second is trivial;
		\item
			we use the upper bound found in Inequality (\ref{upperboundz});
		\item
			we substitute $e_1=\alpha q^{d^2+d-2}$ and, finally, we divide by $q^{d^2-3}$.
	\end{enumerate}
	This yields
	\begin{align*}
		0
			&\le\delta^2(q+4)+\delta q^{d-1}\left(-q^3+\frac{\alpha+7}{2}q^2+\left(\frac{\alpha^2}{2}+4\alpha+10\right)q+(2\alpha+8)\right)\\
			 &\hphantom{\le{}}\mathrel{+}q^{2d}\left((\alpha+1)q+\alpha^2+4\alpha+5\right)+q^{d+1}\left(q^3+3q^2-(\alpha+2)q-(2\alpha+8)\right)
		\end{align*}
		and, using Hypothesis \ref{I} of Theorem \ref{hypo_on_dq} and $\alpha\ge5$, this implies
		\begin{align*}
			0\le\delta^2(q+4)+\delta q^{d+1}\left(\frac{\alpha}{2}+4-q\right)+q^{2d}\left((\alpha+1)q+\alpha^2+5\alpha\right)+q^{d+3}(q+3).
	\end{align*}
	Let $f(\delta)$ denote the right hand side of this equation and set $\delta_1:=(\alpha+3)q^{d-1}-q$ as well as $\delta_2:=q^{d+1}-\left(\frac{\alpha}{2}+8\right)q^d$. We want to show that $\delta$ does not lie in the interval $[\delta_1,\delta_2]$. To see this it suffices to show that $f(\delta_1)<0$ and $f(\delta_2)<0$ hold; the reason for that is, that $f$ is a quadratic polynomial in $\delta$ with positive leading coefficient. Simple calculations show
	\begin{align*}
		f(\delta_1)
			 &=-2q^{2d+1}+\left(\frac{3}{2}\alpha^2+\frac{21}{2}\alpha+12\right)q^{2d}+(\alpha^2+6\alpha+9)q^{2d-1}\\
			 &\hphantom{\le{}}\mathrel{+}(4\alpha^2+24\alpha+36)q^{2d-2}+q^{d+4}+4q^{d+3}-\left(\frac{\alpha}{2}+4\right)q^{d+2}\\
			&\hphantom{\le{}}\mathrel{-}(2\alpha+6)q^{d+1}-(8\alpha+24)q^d+q^3+4q^2
\end{align*}
as well as
	\begin{align*}
			f(\delta_2)
			&=-(\alpha+31)q^{2d+1}+(2\alpha^2+37\alpha+256)q^{2d}+q^{d+4}+3q^{d+3},
	\end{align*}
	and in view of $q\ge \frac32\alpha^2+\frac{21}2\alpha+17$ from Hypothesis \ref{I} of Theorem \ref{hypo_on_dq} both of these are negative.
	Hence, $\delta\notin[\delta_1, \delta_2]$. Finally, we have
	\begin{align*}
		\delta
			 &\mathrel{\makebox[\widthof{\scriptsize\ref{hypo_on_dq}~\ref{I}}][c]{=}}\theta_{d+1}-q-c_0\overset{\text{\scriptsize\ref{ci_properties}~\ref{ci_properties_IisLarge}}}{\le}\frac12(\theta_{d+1}-q)+c_1+\frac12c_3\\
			 &\overset{\makebox[\widthof{\scriptsize\ref{hypo_on_dq}~\ref{I}}][c]{\text{\scriptsize\ref{LemmaWithBounds}~\ref{boundsgauss_simple}}}}{\le}\frac12(q^{d+1}+(q+2)q^{d-1}-q)+c_1+\frac12c_3\\
			 &\overset{\makebox[\widthof{\scriptsize\ref{hypo_on_dq}~\ref{I}}][c]{\text{\scriptsize\ref{ci_properties}~\ref{ci_properties_boundc1c3}}}}{\le}\frac12(q^{d+1}+(q+2)q^{d-1}-q)+2(q+4+\alpha)q^{d-1}\\
			 &\overset{\text{\scriptsize\ref{hypo_on_dq}~\ref{I}}}{<}q^{d+1}-\left(\frac\alpha2+8\right)q^d=\delta_2.
	\end{align*}
	In  the last step we also used $\alpha\ge 5$. From $\delta<\delta_2$ and $\delta\notin[\delta_1,\delta_2]$ we find $\delta<\delta_1$, as claimed.
\end{proof}

\begin{Theorem}\label{Czerothm}
	We have $\F=C_0$.
\end{Theorem}
\begin{proof}
	Put $\delta:=c_1+c_2+c_3$ and note that $c_0+c_1+c_2+c_3=\theta_{d+1}-q$ implies $\delta=\theta_{d+1}-q-c_0$. The strategy of the proof is to show that $\delta=0$ by using Lemma \ref{ci_properties}~\ref{ci_properties_sizeofM} and by finding a good upper bound for $|M|$. From Lemma \ref{ci_properties}~\ref{ci_properties_sizeofW} we have $|W|=q+\delta$ and from
	Lemma \ref{Watmostquadraticinq} we have $|W|\le (\alpha+3)q^{d-1}$. Therefore, for all $F\in C_1\cup C_2$ we have
	\begin{align*}
		|F\cap M|
			&\overset{\text{\scriptsize\ref{HowCimeetsM}}}{\le}
			 (\alpha+3)q^{d^2-1}\theta_{d-1}+\Delta\overset{{\text{\scriptsize\ref{LemmaWithBounds}~\ref{boundsgauss_simple}}}}{\le}(\alpha+3)(q+2)q^{d^2+d-3}+\Delta\\
			 &\overset{\makebox[\widthof{\scriptsize\ref{HowCimeetsM}}][c]{\text{\scriptsize(\ref{boundsDelta})}}}{\le}(\alpha+3)(q+2)q^{d^2+d-3}+(q+2)q^{d^2-1}\overset{\text{\scriptsize\ref{hypo_on_dq}~\ref{I}}}{\le}(\alpha+4)q^{d^2+d-2}.
	\end{align*}
	Since $e_1=\alpha q^{d^2+d-2}$, we know from Lemma \ref{ci_properties} (a) that $|F\cap M|\le(\alpha+4)q^{d^2+d-2}$ for all $F\in C_1\cup C_2\cup C_3$. Therefore, the total contribution from the independent sets in $C_1\cup C_2\cup C_3$ to $M$ is at most $\delta(\alpha+4)q^{d^2+d-2}$.

	By Lemma \ref{HowCimeetsM}~\ref{HowC0meetsM}, the generic parts of the independent sets in $C_0$ are disjoint from $M$. In order to determine the contribution of the sets of $C_0$ to $M$, it remains to consider the special parts $T$ of independent sets $F\in C_0$ and we denote by
	\begin{enumerate}[label=$\bullet$]
		\item $\omega_1$ the number of those with $T$ based on a line that is contained in $U$,
		\item $\omega_2$ the number of those with $T$ based on a line that is not contained in $U$,
		\item $\omega_3$ the number of those with $T$ based on a hyperplane of $\PG(2d,q)$,
		\item $\omega_4$ the number of the remaining ones, which, according to Lemma \ref{Sizeofindependentsets}, are those with $|T|\le q^{d-1}\theta_1\theta_{d-2}\theta_{d-1}^{d-1}$.
	\end{enumerate}
	Furthermore, we let $\Omega_1$ be the set of lines $\ell$ of $U$ such that $F(P,\ell)\in C_0$ for some point $P$ of $\ell$, we let $\Omega_3$ be the set of all point-hyperplane pairs $(P,H)$ with $F(P,H)\in C_0$ such that $U$ is not contained in $H$, and we let $\Omega_4$ be the set of indices $i\in I$ such that $F_i$ is an element of $C_0$ and its special part $T$ has size at most $q^{d-1}\theta_1\theta_{d-2}\theta_{d-1}^{d-1}$. Finally, we let $\widehat{\Omega}_4$ be the set of all flags $f\in M$ such that $f$ is an element of the special part of some independent set $F_x$ with $x\in\Omega_4$.

	Then we have $\omega_1+\omega_2+\omega_3+\omega_4=c_0$, $|\Omega_1|\le\omega_1$, $|\Omega_3|=\omega_3$ and $|\Omega_4|=\omega_4$. In view of the definition of $\Omega_3$ we notice that hyperplane based special parts $T$ only contribute to $M$ when the underlying hyperplane of $\PG(2d,q)$ does not contain $U$.

	Altogether, using Lemma \ref{newlemma} and Lemma \ref{ci_properties}~\ref{ci_properties_sizeofM}, it follows that
	\begin{align}
		|W|q^{d^2-1}\theta_d
			=|M|&\le\delta(\alpha+4)q^{d^2+d-2}+\sum_{ l\in \Omega_1}| l\cap W|q^{d^2-1}\nonumber\\
			&\hphantom{\le{}}\mathrel{+}\omega_2 |W|q^{d^2-d-1}+\sum_{\mathclap{(P,H)\in \Omega_3}}|W|q^{d^2-d}+|\widehat{\Omega}_4|.\label{eqnsmallWW}
	\end{align}
	We proceed by replacing the first sum in the right hand side of the previous inequality by an upper bound. We have
	\begin{align*}
		0
			&\le\sum_{ l\in \Omega_1}(| l\cap W|-1)(| l\cap W|-2)\\
			&=\sum_{ l\in \Omega_1}| l\cap W|(| l\cap W|-1)-2\sum_{l\in \Omega_1}| l\cap W|+2|\Omega_1|\\
			&\le|W|(|W|-1)-2\sum_{l\in \Omega_1}| l\cap W|+2|\Omega_1|,
	\end{align*}
	where the last inequality holds, since any pair of distinct points of $W$ is contained in at most one line of $\Omega_1$. Since $|\Omega_1|\le\omega_1$ and $\omega_1+\omega_2+\omega_3+\omega_4=c_0=\theta_{d+1}-|W|$ (by part \ref{ci_properties_sizeofW} of Lemma \ref{ci_properties}), it follows that
	\begin{align*}
		\sum_{l\in \Omega_1}| l\cap W|\le \frac12|W|(|W|-3)+\theta_{d+1}-\omega_2-\omega_3-\omega_4.
	\end{align*}
	Using this as well as $|\Omega_3|=\omega_3$ in (\ref{eqnsmallWW}) and dividing by $q^{d^2-d}$ we find
	\begin{align*}
		|W|q^{d-1}\left(\theta_d-\frac{|W|-3}{2}\right)
			&\le\delta(\alpha+4)q^{2d-2}+\theta_{d+1}q^{d-1}+\omega_2\frac{|W|-q^d}{q}\\
			&\hphantom{\le{}}\mathrel{+}\sum_{\mathclap{(P,H)\in \Omega_3}}(|W|-q^{d-1})+\frac{|\widehat{\Omega}_4|-\omega_4 q^{d^2-1}}{q^{d^2-d}}.
	\end{align*}
	Since $|W|\le(\alpha+3)q^{d-1}$ by Lemma \ref{Watmostquadraticinq}, the coefficient of $\omega_2$ in the above inequality is not positive and therefore the term with $\omega_2$ can be omitted. Doing so and substituting then $|W|=\delta+q$ we find
	\begin{align}
		(\delta+q)q^{d-1}\underbrace{\left(\theta_d-\frac{\delta+q-3}{2}\right)}_{\ge\theta_1q^{d-1}-\delta}
			&\le\delta(\alpha+4)q^{2d-2}+\theta_{d+1}q^{d-1}\nonumber\\
			&\hphantom{\le{}}\mathrel{+}\sum_{\mathclap{(P,H)\in \Omega_3}}(|W|-q^{d-1})+\frac{|\widehat{\Omega}_4|-\omega_4 q^{d^2-1}}{q^{d^2-d}}.\label{eqnprofinal}
	\end{align}
	We next show that
	\begin{align}
L:=\sum_{(P,H)\in \Omega_3}(|W|-q^{d-1})\le     \delta q^{d+1}.
 \end{align}
 This is trivial, if $|W|\le q^{d-1}$. If $W>q^{d-1}$, then we use $|\Omega_3|\le |C_0|=\theta_{d+1}-|W|$ to see that (where $y:=\frac12(\theta_d+q^{d-1})$)
	\begin{align}\label{eqnkiujh}
		L & \le (\theta_{d+1}-|W|)(|W|-q^{d-1})\\
          & = (|W|-q)q^{d+1}-(|W|-y)^2+y^2-\theta_{d+1}q^{d-1}+q^{d+2}
	\end{align}
	From Lemma \ref{LemmaWithBounds}~\ref{boundsgauss_simple} we find $y\le \frac12q^{d-1}(q+3)$ and, using  Hypothesis \ref{I} of Theorem \ref{hypo_on_dq}, this implies that $y^2-\theta_{d+1}q^{d-1}+q^{d+2}<0$. Therefore \eqref{eqnkiujh} shows that $L\le(|W|-q)q^{d+1}$. Since $|W|=q+\delta$, this establishes $|L|\le \delta q^{d+1}$ in any case.
	Using this and $\theta_{d+1}\le(q^2+q+2)q^{d-1}$ (see Lemma \ref{LemmaWithBounds}~\ref{boundsgauss_simple}) in Inequality (\ref{eqnprofinal}) we find
	\begin{align*}
		(\delta+q)q^{d-1}(\theta_1q^{d-1}-\delta)
			&\le\delta(\alpha+4)q^{2d-2}+(q^2+q+2)q^{2d-2}\\
			&\hphantom{\le{}}\mathrel{+}\delta q^{d+1}+\frac{|\widehat{\Omega}_4|-\omega_4 q^{d^2-1}}{q^{d^2-d}},
	\end{align*}
	which is equivalent to
	\begin{align}
		0	&\le\delta^2q^{d-1}-\delta q^d(q^{d-1}-(\alpha+3)q^{d-2}-q-1)\nonumber\\
			&\hphantom{\le{}}\mathrel{+}2q^{2d-2}+\frac{|\widehat{\Omega}_4|-\omega_4 q^{d^2-1}}{q^{d^2-d}},\label{eqn_last}
	\end{align}

	Finally, we study the cardinality of $\widehat{\Omega}_4$. For all $x\in \Omega_4$ we know from Lemma \ref{Sizeofindependentsets} that $F_x=F(P_x,\mathcal{U}_x)$ for a set $\mathcal{U}_x$ of subspaces of dimension $d$ with 
\begin{align*}
|\mathcal{U}_x|\le (1+q^{-1})\theta_{d-2}\theta_{d-1}^{d-1}.
 \end{align*}
 Furthermore, for all $x\in \Omega_4$ every flag $(\pi,\tau)$ of the special part of $F_x$ that lies in $\widehat{\Omega}_4$ satisfies $\dim(\tau\cap U)=1$ and $\pi\cap U$ is a point of $\tau\cap W$. Motivated by this, we define
	\begin{align*}
		\forall x\in \Omega_4:\zeta_x:=\max\{|\tau\cap W|:\tau\in\mathcal{U}_x,\dim(\tau\cap U)=1\}.
	\end{align*}
	We put $\zeta:=\max\{\zeta_x:x\in \Omega_4\}$ if $\Omega_4\not=\emptyset$, and otherwise we put $\zeta:=q$. There are two remarks to be made.
	\begin{enumerate}[label=$\bullet$]
		\item
			For all $x\in \Omega_4$ and all $\tau\in\mathcal{U}_x$ with $\dim(\tau\cap U)=1$ we have $W\not\ni P_x\in \tau\cap U$, which implies $\zeta_x\le q$. Thus we also have $\zeta\le q$.
		\item
			The definition of $\zeta$ implies that there is a line $l\le U$ with $|l\cap W|=\zeta$. For all $x\in \Omega_4$ with $P_x\in l$ and all $\tau\in\mathcal{U}_x$ with $\dim(\tau\cap U)=1$
 we have $|\tau\cap W|\le\zeta$. For all $x\in \Omega_4$ with $P_x\notin l$ and all $\tau\in\mathcal{U}_x$ with $\dim(\tau\cap U)=1$ we have $|\tau\cap l|\le1$, which implies $|\tau\cap W|\le\min\{\zeta,|W|-(\zeta-1)\}$.
	\end{enumerate}
	This implies
	\begin{align*}
		|\widehat{\Omega}_4|
			&\le(q\zeta+\max\{0,\omega_4-q\}\cdot\min\{\zeta,|W|+1-\zeta\}) q^{d-1}\left(1+q^{-1}\right)\theta_{d-2}\theta_{d-1}^{d-1}\\
			&\le(q^2+\omega_4\cdot\min\{\zeta,|W|+1-\zeta\})(q+1)(q+2)(q+d)q^{d^2-5},
\\
			&\le q^{d^2+1}+\omega_4\cdot\min\{\zeta,|W|+1-\zeta\}(q+1)(q+2)(q+d)q^{d^2-5}.
	\end{align*}
	The second step uses parts \ref{boundsgauss_simple} and \ref{boundsgauss_potenz} of Lemma \ref{LemmaWithBounds}. Substituting this in Equation (\ref{eqn_last}) and dividing by $q^{d-5}$ yields
	\begin{align}\label{eqnnewklaus}
		0	&\le\delta^2q^{4}-\delta q^5(q^{d-1}-(\alpha+3)q^{d-2}-q-1)+2q^{d+3}+q^{6}\nonumber\\
			&\hphantom{\le{}}\mathrel{+}\omega_4(\min\{\zeta,|W|+1-\zeta\}(q+1)(q+2)(q+d)-q^4)
	\end{align}
Since $\min\{\zeta,|W|+1-\zeta\}\le\zeta\le q$ and $\omega_4\le|C_0|\le|\F|=\theta_{d+1}-q\le (q+2)q^d$ we find
	\begin{align*}
		0	&\le\delta^2q^{4}-\delta q^5(q^{d-1}-(\alpha+3)q^{d-2}-q-1)+2q^{d+3}+q^{6}\\
			&\hphantom{\le{}}\mathrel{+}(q+2)(q(q+1)(q+2)(q+d)-q^4)q^d.
	\end{align*}
If we denote the right hand side of this inequality by $g(\delta)$, then $g$ is polynomial in $\delta$ of degree 2 with positive leading coefficient. For $\delta_1:=d+4$ and $\delta_2:=(\alpha+3)q^{d-1}$ we have
	\begin{align*}
		g(\delta_1)
			&=-q^{d+1}(q^3-(\alpha d+8d+4\alpha+22)q^2-(8d+4)q-4d)\\
			&\hphantom{={}}\mathrel{+}(d+5)q^6+(d+4)q^5+(d^2+8d+16)q^4,\\
		g(\delta_2)
			&=-q^{2d+2}((\alpha+3)q-2\alpha^2-12\alpha-18)\\
			&\hphantom{={}}\mathrel{+}(\alpha+3)q^{d+5}+(d+\alpha+6)q^{d+4}\\
			&\hphantom{={}}\mathrel{+}(5d+10)q^{d+3}+(8d+4)q^{d+2}+4dq^{d+1}+q^6.
	\end{align*}
Using Hypothesis \ref{I} of Theorem \ref{hypo_on_dq}, it follows that $g(\delta_1)<0$ and $g(\delta_2)<0$ and hence $g(x)<0$ for all real numbers $x$ with $\delta_1\le x\le\delta_2$. Since $g(\delta)\ge 0$ and $\delta=|W|-q\le (\alpha+3)q^{d-1}$, it follows that $\delta<\delta_1$, that is $\delta\le d+3$.

Put $\beta:=\min\{\zeta,|W|+1-\zeta\}$. Then $|W|+1\ge 2\beta$. Since $|W|=q+\delta$ and $\delta\le d+3$, it follows that $2\beta\le q+d+4$. Using Hypothesis \ref{I} of Theorem \ref{hypo_on_dq} it follows that $\beta\le q-d-3$. This implies that the coefficient of $\omega_4$ in \eqref{eqnnewklaus} is negative. Hence
	\begin{align*}
		0	&\le-\delta q^4(q^d-\delta-(\alpha+3)q^{d-1}-q^2-q)+2q^{d+3}+q^{6}
	\end{align*}
    and since $\delta=|W|-q\le(\alpha+3)q^{d-1}$ (Lemma \ref{Czerothm}) we find $\delta<1$ from Hypothesis \ref{I} of Theorem \ref{hypo_on_dq}. Hence $\delta=0$, that is, $|C_0|=\theta_{d+1}-q-\delta=\theta_{d+1}-q=|\F|$ and thus $C_0=\F$.
\end{proof}

Note that this theorem concludes the proof of Theorem \ref{hypo_on_dq} and it only remains to prove Theorem \ref{main1}.

\textbf{Proof of Theorem \ref{main1}:} Suppose that $d$ is such that Conjecture \ref{conject} holds with $\alpha=\max\{5,\rho(d)\}$ and suppose that $q>3\cdot 112^{2^{d+1}-1}\cdot2^{-d-1}$ as well as $q\ge\frac{3}{2}\alpha^2+\frac{21}{2}\alpha+17$. Consider a coloring of the Kneser graph $\Gamma_d$, with $t\le \theta_{d+1}-q$ color classes $C_1,\dots,C_t$. Define $C_i:=\emptyset$ for $t<i\le\theta_{d+1}-q$. Each set $C_i$ is an independent set of flags of type $\{d-1,d\}$ in $\PG(2d,q)$. If $|C_i|>e_1=\alpha q^{d^2+d-2}$, then let $\bar C_i$ be a maximal independent set containing $C_i$; it follows from Conjecture \ref{conject} that $\bar C_i$ is one of the sets defined in Example \ref{Examples_IndependentSets} and thus we have $g_0\le|\bar C_i|\leq e_0$. For each $i$ we now define a set $F_i$ and for all $i$ with $|C_i|\le e_1$ we simply set $F_i:=C_i$. Now, consider an index $i$ with $|C_i|>e_1$. If there exists an index $j<i$ with $|C_j|>e_1$ and such that $\bar C_i$ and $\bar C_j$ have the same generic part, then let $F_i$ be the special part of $\bar C_i$ (this implies $|F_i|\leq q^d\gauss{2d-1}{d-1} <e_1$) and otherwise set $F_i:=\bar C_i$. Let $S$ be the set of indices $i$ with $|F_i|>e_1$. Consider the multiset $\F=\{F_i\mid 1\le i\le\theta_{d+1}-q\}$. Then each $F_i$ is an independent set and the union of the $F_i$ is the set of all flags of type $\{d-1,d\}$ in $\PG(2d,q)$. We consider two cases.

Case 1. For at least $\frac12|S|$ indices $i\in S$ the generic part of $F_i$ is based on a point. Then $\F$ satisfies all hypotheses of Theorem \ref{hypo_on_dq}. The conclusion of this theorem implies $S=\{1,2,\dots,\theta_{d+1}-q\}$, we know that the generic part of every set $F_i$ is based on a point and the base points are $\theta_{d+1}-q$ distinct points of a subspace of dimension $d+1$. This implies $t=\theta_{d+1}-q$ as well as $|C_i|>e_1$ and $F_i=\bar C_i$ for all $i$. Notice that $F_i=\bar C_i$ might not be uniquely determined by $C_i$, however its base point is. This follows from the fact that two maximal independent sets based on distinct points (are easily seen to) have less than $e_1$ elements in common and hence $C_i$ can not be contained in both. This proves Theorem \ref{main1} in this case.

Case 2. For less than $\frac12|S|$ indices $i\in S$ the generic part of $F_i$ is based on a point. Then for more than $\frac12|S|$ indices $i$ the generic part is based on a hyperplane and we can apply the first case in the dual space. This proves Theorem \ref{main1}.

\end{subsection}
\end{section}

\section*{Acknowledgement}
The research of Jozefien D'haeseleer is supported by the FWO (Research Foundation Flanders).


\begin{thebibliography}{11}
	\bibitem{blokhuis1}
		Blokhuis, A., Brouwer, A.E., Cocliques in the Kneser graph on line-plane flags in PG(4,q). \textit{Combinatorica 37}: 795--804, (2017).\\
		\url{https://doi.org/10.1007/s00493-016-3438-2}
	\bibitem{blokhuis_neu}
		Blokhuis, A., Brouwer, A. E., Chowdhury, A., Frankl, P., Mussche, T., Patk\'{o}s, B., Sz\H{o}nyi, T., {A Hilton-Milner theorem for vector spaces}. \textit{Electron. J. Combin. 17}, (2010).\\
		\url{https://doi.org/10.37236/343}
	\bibitem{blokhuis3}
		Blokhuis, A., Brouwer, A. E., Szönyi, T., On the chromatic number of $q$-Kneser graphs. \textit{Des. Codes Cryptogr. 65}: 187-197, (2011)\\
		\url{https://doi.org/10.1007/s10623-011-9513-1}
	\bibitem{Sam}
		De Beule, J., Metsch, K., Mattheus, S., An algebraic approach to Erdos-Ko-Rado sets of flags in spherical buildings. arXiv:2007.01104.
	\bibitem{Jozefien&Daniel&Klaus}
		D'haeseleer, J., Metsch, K., Werner, D., On the chromatic number of two generalized Kneser graphs, arXiv:2005.05762.
	\bibitem{erdos_ko_rado}
		Erd{\H{o}}s P., Ko C., Rado R., Intersection Theorems for systems of finite sets. \textit{Quart. J. Math. Oxford Ser. 12(2)}: 313--320, (1961).
	\bibitem{Hirschfeld}
		Hirschfeld, J.~W.~P., {\it Projective Geometries over Finite Fields}, Second edition, Oxford University Press, Oxford (1998).
	\bibitem{Ferdinand&Klaus}
		Ihringer, F., Metsch, K., Large {$\{0,1,\dots,t\}$}-cliques in dual polar graphs. \textit{J. Combin. Theory Ser. A 154}: 285--322, (2018).
	\bibitem{Daniel&Klaus}
		Metsch, K., Werner, D., Maximal co-cliques in the Kneser graph on plane-solid flags in PG(6,q). \textit{Innov. Incidence Geom.}, 18 (2020), no.1, 39--55.




\end{thebibliography}
\end{document}